\documentclass{article}
\usepackage{graphicx}
\usepackage{setspace}
\usepackage{amsmath}
\usepackage{amssymb,amsfonts,amsbsy}
\usepackage{amsthm}
\usepackage[pdftex]{hyperref}
\usepackage{url}

\newcommand{\R}{\mathbb{R}}
\newcommand{\N}{\mathbb{N}}
\newcommand{\Z}{\mathbb{Z}}

\newcommand{\half}{\tfrac{1}{2}}

\newtheorem{theorem}{Theorem}[section]
\newtheorem{proposition}[theorem]{Proposition}
\newtheorem{lemma}[theorem]{Lemma}

\newtheorem{conjecture}[theorem]{Conjecture}

\begin{document}
\title{Asymptotics of a discrete-time particle system near a reflecting boundary}
\author{Jeffrey Kuan}
\maketitle
\begin{abstract} We examine a discrete-time Markovian particle system on $\N\times\Z_+$ introduced in \cite{kn:D}. The boundary $\{0\}\times\Z_+$ acts as a reflecting wall. The particle system lies in the Anisotropic Kardar-Parisi-Zhang with a wall universality class. After projecting to a single horizontal level, we take the long--time asymptotics and obtain the discrete Jacobi and symmetric Pearcey kernels. This is achieved by showing that the particle system is identical to a Markov chain arising from representations of $O(\infty)$ (introduced in \cite{kn:BK}). The fixed--time marginals of this Markov chain are known to be determinantal point processes, allowing us to take the limit of the correlation kernel.

We also give a simple example which shows that in the multi-level case, the particle system and the Markov chain evolve differently.
\end{abstract}

\section{Introduction}
In the study of random interface growth, universality is a ubiquitous topic. Informally, universality says that random growth models with similar physical properties will have identical behavior at long--time asymptotics. In particular, different models in the same universality class are expected to have the same growth exponents and limiting distributions. In this sense, the classical central limit theorem is a universality statement, where the growth exponent is $1/2$ and the limiting distribution is Gaussian, regardless of the distribution of each summand.

A different universality class, called the Kardar-Parisi-Zhang (KPZ) universality class (introduced in \cite{KPZ}), models a variety of real--world growth processes, such as turbulent liquid crystals \cite{TS} and bacteria colony growth \cite{WIMM}. If $h(\vec{x},t)$ is the height of the interface at location $\vec{x}$ and time $t$, then it satisfies the stochastic differential equation
$$
\frac{\partial h}{\partial t}=\nu \nabla^2 h + \frac{\lambda}{2} (\nabla h)^2 + \eta(\vec{x},t),
$$
where $\eta(\vec{x},t)$ is space--time white noise. Due to the non--linearity, however, this stochastic differential equation is not well--defined. A common mathematical approach has been to study exactly solvable models (i.e. where the finite--time probability distributions can be computed exactly)  in the universality class, and then to take the long--time limits. Examples of such models include random matrix theory \cite{TW2}, the PNG droplet \cite{PS}, ASEP \cite{TW}, non--intersecting Brownian motions \cite{ABK}, and random partitions \cite{kn:BK0}. In all of these models, the growth exponent is $1/3$ and the limiting distribution is called the Airy process, demonstrating the universality of the KPZ equation. More recently, there have also been mathematically rigorous interpretations of a solution to the KPZ equation \cite{ACQ,H}.

The universality class considered in this paper is called anisotropic Kardar-Parisi-Zhang (AKPZ) with a wall. It is a variant of KPZ in two ways: there is anisotropy and the substrate acts as a reflecting barrier. As before, the stochastic differential equation is not well--defined, so we take the approach of analyzing exactly solvable models. So far, there have only been two models which have been proven to be in this universality class: a randomly growing stepped surface in $2+1$ dimensions \cite{kn:BK} and non--intersecting squared Bessel paths \cite{KMW}. In both cases, the limiting behavior near the critical point of the barrier has growth exponent $1/4$ and limiting process the Symmetric Pearcey process (defined in section \ref{SPK}).




The exactly solvable model considered here was introduced in \cite{kn:D}. It is a discrete-time interacting particle system with a wall which evolves according to geometric jumps with a parameter $q\in [0,1)$. In the $q\rightarrow 1$ limit, this model also has connections to a random matrix model. The main result of this paper is that in the long--time asymptotics near the wall, the symmetric Pearcey process appears after rescaling by $N^{1/4}$. This therefore helps to establish the universality of the growth exponent $1/4$ and the Symmetric Pearcey process in the AKPZ with a wall universality class. The approach is to show that when projected to a single level and to (finite) integer times, the particle system is identical to a previously studied family of determinantal point process. Taking asymptotics of the correlation kernel then yields the desired results. 

We will also show that away from the critical point and at finite distances from the wall, the discrete Jacobi kernel appears in the long--time asymptotics. This kernel also appeared in the long--time limit in \cite{kn:BK}, but has not appeared anywhere else. In particular, it did not appear in non--intersecting squared Bessel paths \cite{KMW0}.


In section \ref{2}, we review the particle system from \cite{kn:D} and the determinantal point processes from \cite{kn:BK}. In section \ref{3}, we compute the correlation kernel for the particle system on one level by showing that the two processes are identical. In section \ref{4}, we take the large-time asymptotics.

The models in \cite{kn:D} and \cite{kn:BK} have connections to the representation theory of the orthogonal groups, but this paper is intended to be understandable without knowledge of representation theory.

It should also be true that given the initial conditions, the fixed-time distributions for the two models are identical without needing to restrict to a single level, but this is not pursued here.

\textbf{Acknowledgements}. The author would like to thank Alexei Borodin, Manon Defosseux, Ivan Corwin and the referees for helpful comments. 

\section{Two Models}\label{2}
\subsection{Interacting Particle System}\label{IPS}
The interacting particle system in \cite{kn:D} arises from a Pieri-type formula for the (finite-dimensional) orthogonal groups. Here, we briefly describe the model. 
 
The particles live on the lattice\footnote{$\N$ denotes the non--negative integers and $\Z_+$ denotes the positive integers.} $\N\times\Z_+$. The horizontal line $\N\times\{k\}$ is often called the $k$th level. There are always $\lfloor \tfrac{k+1}{2} \rfloor$ particles on the $k$th level, whose positions at time $n$ will be denoted $X^k_1(n)\geq X^k_2(n)\geq X^k_3(n)\geq \ldots\geq X^k_{\lfloor (k+1)/2 \rfloor}(n)\geq 0$. The time can take integer or half--integer values. For convenience of notation, $X^k(n)$ will denote $(X^k_1(n), X^k_2(n), X^k_3(n), \ldots, X^k_{\lfloor (k+1)/2 \rfloor}(n))$ $\in\mathbb{N}^{\lfloor (k+1)/2 \rfloor}$. More than one particle may occupy a lattice point. The particles must satisfy the \textit{interlacing property}
\begin{equation}\label{Interlacing}
X^{k+1}_{i+1}(n)  \leq X^k_i(n) \leq X^{k+1}_i(n)
\end{equation}
for all meaningful values of $k$ and $i$. This will be denoted $X^k\prec X^{k+1}$. With this notation, the state space can be described as the set of all sequences $(X^1\prec X^2 \prec \ldots)$ where each $X^k\in \mathbb{N}^{\lfloor (k+1)/2 \rfloor}$. The initial condition is $X_i^k(0)=0$, called the \textit{densely packed} initial conditions. Now let us describe the dynamics. 

For $n\geq 0,k\geq 1$ and $1\leq i \leq \lfloor \tfrac{k+1}{2} \rfloor$, define random variables
\[
\xi^k_i(n+1/2), \ \ \xi^k_i(n)
\]
which are independent identically distributed geometric random variables with parameter $q$. In other words, $\mathbb{P}(\xi^1_1(1/2)=x)=q^x(1-q)$ for $x\in\N$. Let $R(x,y)$ be a Markov kernel on $\N$ defined by
\[
R(x,y)= \frac{1-q}{1+q}\cdot\frac{q^{\vert x-y\vert}+q^{x+y}}{1+1_{y=0}},
\]
so that $R(x,\cdot)$ is the law of the random variable $\vert x+\xi_1^1(1)-\xi_1^1(\tfrac{1}{2})\vert$.

At time $n$, all the particles except $X^k_{({k+1})/{2}}(n)$ try to jump to the left one after another in such a way that the interlacing property is preserved. The particles $X^k_{({k+1})/{2}}(n)$ do not jump on their own. The precise definition is 
\begin{eqnarray*}
X^k_{(k+1)/{2}}(n+\tfrac{1}{2})&=&\min(X^k_{(k+1)/{2}}(n), X_{(k-1)/2}^{k-1}(n+\tfrac{1}{2}))\ \ k \text{ odd}\\
X^k_i(n+\tfrac{1}{2})&=& \max(X^{k-1}_i(n), \min(X^k_i(n),X^{k-1}_{i-1}(n+\tfrac{1}{2})) - \xi^k_i(n+\tfrac{1}{2})),
\end{eqnarray*}
where $X^{k-1}_0(n+\tfrac{1}{2})$ is formally set to $+\infty$.

At time $n+\tfrac{1}{2}$, all the particles except $X^k_{({k+1})/{2}}(n+\half)$ try to jump to the right one after another in such a way that the interlacing property is preserved. The particles $X^k_{({k+1})/{2}}(n+\half)$ jump according to the law $R$. The precise definition is 
\[
X^k_{(k+1)/{2}}(n+1)=\min(\vert X^k_{(k+1)/{2}}(n) + \xi^k_{(k+1)/2}(n+1) - \xi^k_{(k+1)/2}(n+\half) \vert, X^{k-1}_{(k-1)/{2}}(n) )
\]
when $k$ is odd and
\[
X^k_i(n+1)= \min(X^{k-1}_{i-1}(n+\half), \max(X^k_i(n+\half),X^{k-1}_{i}(n+1)) + \xi^k_i(n+1)),
\]
where $X^{k-1}_0(n+1)$ is formally set to $+\infty$.

Let us explain the particle system. The particles preserve the interlacing property in two ways: by pushing particles above it, and being blocked by particles below it. So, for example, in the left jumps, the expression $\min(X^k_i(n),X^{k-1}_{i-1}(n+\tfrac{1}{2}))$ represents the location of the particle after it has been pushed by a particle below and to the right. Then the particle attempts to jump to the left, so the term $\xi^k_i(n+\tfrac{1}{2})$ is subtracted. However, the particle may be blocked a particle below and to the left, so we must take the maximum with $X^{k-1}_i(n)$.

While $X_i^k(n)$ is not simple, applying the shift $\tilde{X}_i^k(n)=X_i^k(n) + \lfloor\tfrac{k+1}{2}\rfloor -i$ yields a simple process. In other words, $\tilde{X}$ can only have one particle at each location.

Figure \ref{Jumping} shows an example of $\tilde{X}$. Additionally, an interactive animation can be found at \url{http://www.math.harvard.edu/~jkuan/DiscreteTimeWithAWall.html}

By drawing lozenges around the particles as in Figure \ref{3D}, one can see that the particle system can be interpreted as a two--dimensional stepped surface. This can be made rigorous by defining the height function at a point to be the number of particles to the right of that point. With this interpretation, the jumping of the particles corresponds to adding and removing sticks, and therefore the interacting particle system is equivalent to a randomly growing surface. The anisotropy is shown with the observation that only sticks of one type may be added or removed. The necessity of the interlacing condition is also visually apparent: it guarantees that the lozenges can be drawn in a way to make the figure three--dimensional.

\begin{center}
\begin{figure}
\caption{The figure on the left shows lozenges corresponding to the top figure in Figure \ref{Jumping}. The top right figure shows sticks that are never added or removed with each jump, while the bottom right figure shows sticks that are added or removed.}
\label{3D}
\includegraphics[height=2in]{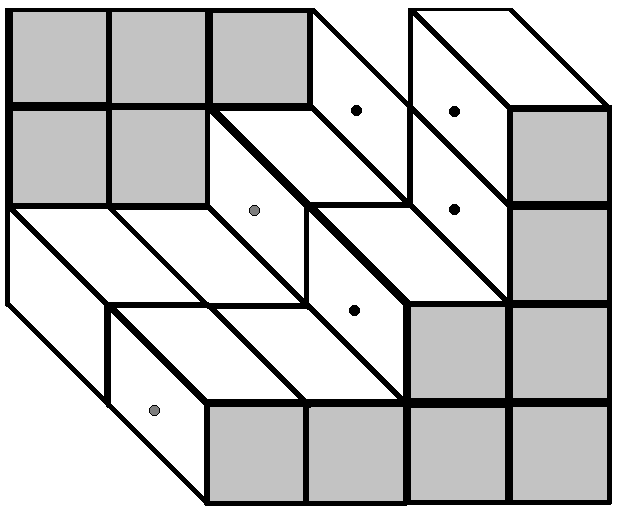}
\quad \quad \quad \quad
\includegraphics[height=2in]{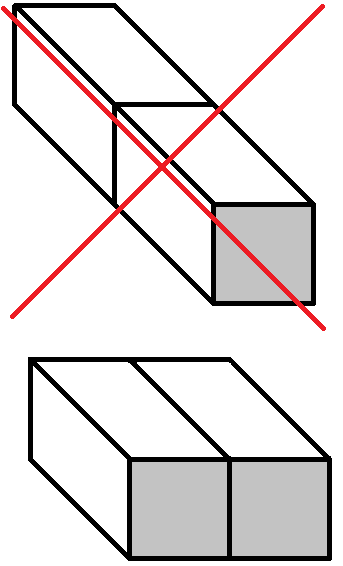}
\end{figure}
\end{center}

\subsection{Determinantal Point Processes}\label{MC}
In \cite{kn:BK}, the authors introduce a family of determinantal point processes, indexed by a time parameter $n\in\mathbb{N}$, which arise from representations of the infinite-dimensional orthogonal group. (See \cite{kn:B} for background on determinantal point processes.) This family depends on a function $\phi\in C^1[-1,1]$. Each determinantal point process also lives on the lattice $\N\times\Z_+$, with exactly $\lfloor \tfrac{k+1}{2}\rfloor$ particles on the $k$th level, and the particles must also satisfy the interlacing property. 

\textbf{Remark on notation.} It is convenient to re-label the levels. For $a=\pm 1/2$, one should think of $(r,a)$ as corresponding to the $2r+a+\half$ level. Throughout this paper, the letter $k$ will denote the level and the letter $r$ will denote the number of particles. Set $\mathbb{J}_r$ to be the set of nonincreasing sequences of integers $(\lambda_1\geq\ldots\geq\lambda_r\geq 0)$. The superscript $\lambda^{(k)}$ will mean that $\lambda^{(k)}$ lives on the $k$th level.  To save space, bold greek letters such as $\boldsymbol{\lambda}$ will denote 
$\boldsymbol{\lambda} = (\lambda^{(1)}\prec\lambda^{(2)}\prec\ldots\prec\lambda^{(k)})$,
and similarly for $\boldsymbol{X}$, and $\mathbf{0}$ will denote the densely packed initial conditions. 

Let $\boldsymbol{Y}(n)$ denote the positions of the particles in this determinantal point process at time $n$. Proposition 3.11 of \cite{kn:BK} establishes that there is a Markov chain\footnote{Strictly speaking, Proposition 3.11 proves \eqref{MP} without showing that $T^{\phi}$ has non--negative entries. A better term would be ``signed Markov chain,'' but this is not standard terminology. In any case, Proposition \ref{Prop} below will show that for the $\phi$ studied in this paper, $T^{\phi}$ is a bona--fide Markov chain.}  $T^{\phi}$  connecting $\boldsymbol{Y}(n)$, in the sense that 
\begin{equation}\label{MP}
\mathbb{P}(\boldsymbol{Y}(n+1)=\boldsymbol{\mu}) = \sum_{\boldsymbol{\lambda}} \mathbb{P}(\boldsymbol{Y}(n)=\boldsymbol{\lambda}) T^{\phi}(\boldsymbol{\lambda,\mu}). 
\end{equation}
Now let us give the formula for $T^{\phi}$. 

Let $\mathsf{J}_s^{(a,b)}(x)$ denote the (normalized) $s$-th Jacobi polynomial with parameters $a,b$. These are polynomials of degree $s$ which are orthogonal with respect to the measure $(1-x)^a(1+x)^bdx$ on $[-1,1]$. In this paper, we just need the equations 
\begin{align*}
\mathsf{J}_s^{(1/2,-1/2)}\left(\frac{z+z^{-1}}{2}\right) =& \frac{z^{s+1/2}-z^{-s-1/2}}{z^{1/2}-z^{-1/2}},\\
\mathsf{J}_s^{(-1/2,-1/2)}\left(\frac{z+z^{-1}}{2}\right) =& \frac{z^s+z^{-s}}{2}.
\end{align*}
Also define
\[
W^{(a,b)}(s)=
\begin{cases}
2,\ \ \text{if}\ \ s>0,a=b=-\frac{1}{2}\\
1,\ \ \text{if}\ \ s=0,a=b=-\frac{1}{2}\\
1,\ \ \text{if}\ \ s\geq 0,a=\frac{1}{2},b=-\frac{1}{2}
\end{cases}
\]

For a function $\phi\in C^1[-1,1]$, define 
\[
I_a^{\phi}(l,s)=\frac{W^{(a,-1/2)}(s)}{\pi}\int_{-1}^1 \mathsf{J}_s^{(a,-1/2)}(x) \mathsf{J}_l^{(a,-1/2)}(x) \phi(x) (1-x)^a(1+x)^{-1/2}dx.
\]
For $a=\pm\half$, define the matrix $T_{r,a}^{\phi}$ with nonnegative entries, and rows and columns paramterized by $\mathbb{J}$:
\[
T_{r,a}^{\phi}(\mu,\lambda)=\det[I_a^{\phi}(\mu_i-i+r,\lambda_j-j+r)]_{1\leq i,j\leq r}\frac{\dim_{2r+1/2+a}\lambda}{\dim_{2r+1/2+a}\mu}.
\]
Here $\dim$ is the dimension of the corresponding representation of $SO(2r+1/2+a)$ -- but for the purposes of this paper, it suffices just to know that $\dim$ is a positive integer. In the proof of Proposition \ref{Prop}, the $\dim$ terms will cancel immediately anyway. Set
\[
T_k^{\phi}=\begin{cases} T^{\phi}_{\lfloor(k+1)/2\rfloor,1/2}, \ \ k\ \ \textit{even} \\ T^{\phi}_{\lfloor (k+1)/2\rfloor,-1/2}, \ \ k\ \ \textit{odd} \end{cases}
\]
For $\lambda$ on the $k$th level and $\mu$ on the $k-1$ level, let $\varkappa_{k-1}^k$ be 
\[
\varkappa^k_{k-1}(\lambda,\mu)=
\begin{cases}
0,\ \ \mu\not\prec\lambda \\
1, \ \ \mu\prec\lambda\ \ \textit{and k odd} \\
1,\  \ \mu\prec\lambda,\ \mu_{r/2}=0, \ \ \textit{and k even}\\
2, \ \ \mu\prec\lambda,\ \mu_{r/2}>0,\ \ \textit{and k even}
\end{cases}
\]
and $T^k_{k-1}$  be
\[
T^k_{k-1}(\lambda,\mu)=\frac{\dim_{k}\mu}{\dim_{k+1}\lambda}\varkappa^k_{k-1}(\lambda,\mu)
\]
and
\[
\Delta^k_{k-1}(\lambda,\mu)=\sum_{\nu} T_k(\lambda,\nu)T^k_{k-1}(\nu,\mu)
\]

The matrix of transition probabilities is 
\[
T^{\phi}(\boldsymbol{\mu},\boldsymbol{\lambda})
=T_1^{\phi}(\mu^{(1)},\lambda^{(1)})\prod_{j=2}^k \frac{T^{\phi}_j(\mu^{(j)},\lambda^{(j)})T^j_{j-1}(\lambda^{(j)},\lambda^{(j-1)})}{\Delta_{j-1}^j(\lambda^{(j)},\lambda^{(j-1)})}.
\]
For the densely packed initial conditions, $T^{\phi}$ satisfies a semingroup property. More specifically, if $T^{\phi_1}T^{\phi_2}$ denotes matrix multiplication, then (\cite{kn:BK})
\[
T^{\phi_1\phi_2}(\mathbf{0},\boldsymbol{\mu})=[T^{\phi_1}T^{\phi_2}](\mathbf{0},\boldsymbol{\mu}).
\] 
When projected to the $k$th level, the matrix of transition probabilities is just $T_k^{\phi}$.

Certain functions $\phi$ arise naturally from the representations of $O(\infty)$ -- see section 2.1 of \cite{kn:BK}. For our purposes, it suffices to consider the function:
$$
\phi_{\alpha}(x)=(1+\alpha(1-x)+\alpha^2(1-x)/2)^{-1}, \ \ \alpha\geq 0.
$$

By Proposition 4.1 from \cite{kn:BK}, $\boldsymbol{Y}(n)$ is determinantal with correlation kernel $K(r_1,a_1,s_1;r_2,a_2,s_2)$ equal to 
\begin{multline}\label{Kernel}
1_{2r_1+a_1\geq 2r_2+a_2}\frac{W^{(a_1,-1/2)}(s_1)}{\pi}\int_{-1}^1\mathsf{J}_{s_1}^{(a_1,-1/2)}(x)\mathsf{J}_{s_2}^{(a_2,-1/2)}(x)(x-1)^{r_1-r_2}(1-x)^{a_1}(1+x)^{-1/2}dx\\
+\frac{W^{(a,-1/2)}(s_1)}{\pi}\frac{1}{2\pi i}\int_{-1}^1\oint_C \frac{\phi_{\alpha}(x)^n}{\phi_{\alpha}(u)^n} \mathsf{J}_{s_1}^{(a_1,-1/2)}(x) \mathsf{J}_{s_2}^{(a_2,-1/2)}(u)\\
\times \frac{(x-1)^{r_1}}{(u-1)^{r_2}}\frac{(1-x)^{a_1}(1+x)^{-1/2}dudx}{x-u}.
\end{multline}

\subsection{Finite--time distributions}
The next theorem, which will be proved in the next section, establishes that $\tilde{X}^k$ is a determinantal point process. 

\begin{theorem}\label{MainTheorem}  Let $\alpha=2q/(1-q)$. Then $\tilde{X}^k(n)=Y^k(n).$ In particular, $\tilde{X}^k(n)$ is a determinantal point process on $\N$ with  kernel $K(r,a,s_1;r,a,s_2)$, where $2r+1/2+a=k$.

\end{theorem}



Numerical calculations made by the author indicate that the fixed time marginals on multiple levels should also be identical, assuming the densely packed initial conditions. The exact statement is below:
\begin{conjecture}\label{Conjecture}
For any time $n\geq 0$,
\[
\mathbb{P}(\boldsymbol{X}(n)=\boldsymbol{\lambda})=T^{\phi_{\alpha}^n}(\boldsymbol{0},\boldsymbol{\lambda}).
\]
\end{conjecture}

Note that without the fixed-time assumption, the conjecture is false. For example, 
\[
\mathbb{P}(\boldsymbol{X}(n+1)=(0,0,(0,0)) \vert \boldsymbol{X}(n) = (0,1,(1,0)))=0,
\]
by the fact that $X_1^2$ prevents $X_1^3$ from jumping to $0$. However, 
\[
T^{\phi_{\alpha}}((0,1,(1,0)),(0,0,(0,0)))\neq 0,
\]
since none of the terms in the definition of $T^{\phi_{\alpha}}$ is zero.

\section{Proof of theorem \ref{MainTheorem}}\label{3}

Let $P_k(\lambda,\beta)$ denote the transition kernel of $X$ on the $k$th level. In other words
\[
P_k(\lambda,\beta)=\mathbb{P}(X^k(n+1)=(\beta_1,\beta_2,\ldots,\beta_{\lfloor (k+1)/2\rfloor}) \vert X^k(n)= (\lambda_1,\lambda_2,\ldots,\lambda_{\lfloor (k+1)/2\rfloor}))
\]
By Theorem 7.1 of \cite{kn:D}, 
\[
P_{2r}(\lambda,\beta)=\sum_{c\in\N^r, c\prec\lambda,\beta} (1-q)^{2r}\frac{\dim_{2r+1}\beta}{\dim_{2r+1}\lambda} q^{\sum_{i=1}^r \lambda_i+\beta_i-2c_i}\left(1_{c_r>0}+\frac{1_{c_r=0}}{1+q}\right)
\]
\[
P_{2r+1}(\lambda,\beta)=\sum_{c\in \N^{r}, c\prec\lambda,\beta}(1-q)^{2r+1}\frac{\dim_{2r+2}\beta}{\dim_{2r+2}\lambda}  q^{\sum_{i=1}^{r} \lambda_i+\beta_i-2c_i} R(\lambda_{r+1},\beta_{r+1}).
\]

In this section, we prove Theorem \ref{MainTheorem}. Set $\phi=\phi_{\alpha}$, where $\alpha=\tfrac{2q}{1-q}$. Since $\boldsymbol{X}(0)=\boldsymbol{Y}(0)=0$, the following proposition suffices.
\begin{proposition}\label{Prop}
For $a=\pm\half$, $T_{k}^{\phi_{\alpha}}=P_{k}$.
\end{proposition}
We start with a few lemmas.

\begin{lemma}\label{FirstLemma} Let 
\[
\phi(x)=\frac{1}{1+\alpha(1-x)+\frac{\alpha^2}{2}(1-x)},\ \ \alpha=\frac{2q}{1-q}.
\]
Then $I_{-1/2}^{\phi}(l,k)=R(l,k)$ and 
\[
I_{1/2}^{\phi}(l,k)=\frac{q-1}{q+1}\left(q^{k+l+1}-q^{\vert k-l \vert}\right)
\]
\end{lemma}
\begin{proof}
Substitute $x=(z+z^{-1})/2$. Then
\[
I_{-1/2}^{\phi}(l,k)=\frac{W^{(-1/2,-1/2)}(k)}{2\pi i}\oint_{\vert z\vert=1}\frac{z^k+z^{-k}}{2} \frac{z^l+z^{-l}}{2}\frac{(1-q)^2}{(1-qz)(z-q)}dz,
\]
which has residues at $q$ and $0$. The residue at $z=q$ is 
\[
W^{(-1/2,-1/2)}(k)\frac{q^k+q^{-k}}{2} \frac{q^l+q^{-l}}{2}\frac{1-q}{1+q}.
\]
Using the expansion
\[
\frac{1}{(1-qz)(z-q)}=\sum_{m=0}^{\infty}\frac{q^{m+1}-q^{-m-1}}{1-q^2}z^m,
\]
the residue at $z=0$ is
\[
W^{(-1/2,-1/2)}(k)\frac{1-q}{1+q}\left(\frac{q^{k+l}-q^{-k-l}}{4} + \frac{q^{\vert k-l\vert}-q^{-\vert k-l\vert}}{4} \right),
\]
so the total contribution is
\[
I_{-1/2}^{\phi}(l,k)=\frac{W^{(-1/2,-1/2)}(k)}{2}\frac{1-q}{1+q}(q^{k+l}+q^{\vert k-l\vert})=R(l,k).
\]
For $a=1/2$,
\[
-\frac{1}{4\pi i}\oint_{\vert z\vert=1}(z^{k+1/2}-z^{-k-1/2}) (z^{l+1/2}-z^{-l-1/2})\frac{(1-q)^2}{(1-qz)(z-q)}dz.
\]
The residue at $z=q$ is 
\[
-\frac{1}{2}(q^{k+1/2}-q^{-k-1/2}) (q^{l+1/2}-q^{-l-1/2})\frac{1-q}{1+q}
\]
and the residue at $z=0$ is
\[
-\frac{1}{2}\frac{1-q}{1+q}\left(q^{k+l+1}-q^{-k-l-1} - q^{\vert k-l\vert} + q^{-\vert k-l\vert} \right).
\]
\end{proof}

\begin{lemma}\label{Interlace}
Let  $c=(c_1\geq c_2\geq \ldots \geq c_r)$ and $\lambda=(\lambda_1\geq \lambda_2 \geq \ldots \geq \lambda_r)$. Set  
\[
\psi_m(s,l)=
\begin{cases}
m, \ \ \text{if}\ l\geq s=0\\
1, \ \ \text{if}\ l\geq s>0\\
0, \ \ \text{if}\ l<s.
\end{cases}
\]
Then 
\[
\det[\psi_m(c_i-i+r,\lambda_j-j+r)]=
\begin{cases}
m, \ \ \text{if}\ c\prec\lambda,\ c_r=0\\
1, \ \ \text{if}\ c\prec\lambda,\ c_r>0\\
0, \ \ \text{if}\ c\not\prec\lambda.
\end{cases}
\]


\end{lemma}
\begin{proof}
The proof is standard, see e.g. Lemma 3.8 of \cite{kn:BK}.
\end{proof}


Now return to the proof of Theorem \ref{MainTheorem}. Start with the odd case. By the lemma, we can write
\[
P_{2r+1}(\lambda,\beta)=(1-q)^{2r}\frac{\dim_{2r+1}\beta}{\dim_{2r+1}\lambda}\sum_{s_1>\ldots>s_r\geq 0} \det[f_{s_i,1}(\lambda_j-j+r)] \det[f_{s_i,\frac{1}{1+q}}(\beta_j-j+r)],
\]
where
\[
f_{s,m}(l)=q^{l-s}\psi_m(s,l).
\]
Thus, by Lemma 2.1 of \cite{kn:BK}, 
\[
P_{2r+1}(\lambda,\beta)=(1-q)^{2r}\frac{\dim_{2r+1}\beta}{\dim_{2r+1}\lambda}\det\left[\sum_{s=0}^{\infty}f_{s,1}(\lambda_i-i+r)f_{s,\frac{1}{1+q}}(\beta_j-j+r)\right].
\]
A simple calculation shows that 
\[
\sum_{s=0}^{\infty}f_{s,1}(x)f_{s,\frac{1}{1+q}}(y)=
\begin{cases}
\dfrac{1}{1+q}q^{x+y}, \ \text{if}\ \ \min(x,y)=0 ,  \\
\dfrac{q^{x+y+1}-q^{\vert x-y\vert}}{q^2-1}, \ \text{otherwise},
\end{cases}
\]
which, by Lemma \ref{FirstLemma}, equals $(1-q)^{-2}I_{1/2}^{\phi}(x,y)$.

Now proceed to the even case. Lemma 2.1 from \cite{kn:BK} is not immediately applicable, because we are summing over elements of $\mathbb{N}^{r-1}$ while the determinants are of size $r$. Notice, however, that $c\prec\lambda,\beta$ if and only if $c\prec\lambda_{\text{red}},\beta_{\text{red}}$ (where $\lambda_{\text{red}},\beta_{\text{red}}$ denote $(\lambda_1,\ldots,\lambda_{r-1}),(\beta_1,\ldots,\beta_{r-1})$) and $c_{r-1}\geq\max(\lambda_r,\beta_r)$. Thus
\begin{multline*}
(1-q)^{2r-1}\frac{\dim_{2r}\beta}{\dim_{2r}\lambda}\frac{q^{\vert\lambda_r-\beta_r\vert}+q^{\lambda_r+\beta_r}}{1+1_{\beta_r=0}}\\
\times\sum_{s_1>s_2>\ldots> s_{r-1}\geq\max(\lambda_r,\beta_r)} \det[f_{s_i,1}(\lambda_j-j+r-1)]_1^{r-1}\det[f_{s_i,1}(\beta_j-j+r-1)]_1^{r-1}\\
=(1-q)^{2r-2}R(\lambda_r,\beta_r)\frac{\dim_{2r}\beta}{\dim_{2r}\lambda}\det\left[\sum_{s=\max(\lambda_r,\beta_r)}^{\infty}f_{s,1}(\lambda_i-i+r-1)f_{s,1}(\beta_j-j+r-1)\right]_1^{r-1}.
\end{multline*}
A straightforward calculation shows that if $\max(\lambda_r,\beta_r)\leq\min(x,y)$, then
\[
\sum_{s=\max(\lambda_r,\beta_r)}^{\infty}f_{k,1}(x)f_{k,1}(y)=\frac{q^{x+y-2\max(\lambda_r,\beta_r)+2}-q^{\vert x-y\vert}}{q^2-1}.
\]
To deal with the case $\max(\lambda_r,\beta_r)>\min(x,y)$, we use the following lemma.
\begin{lemma}
If $\max(\lambda_r,\beta_r)>\min(\lambda_{r-1},\beta_{r-1})$, then $P_{2r-1}(\lambda,\beta)=T^{\phi}_{2r-1}(\lambda,\beta)=0$.
\end{lemma}
\begin{proof}
The fact that $P_{2r-1}(\lambda,\beta)=0$ follows immediately from the description of the interacting particle system, or from the fact that $\{c\in\mathbb{N}^{r-1}:c\prec\lambda,\beta\}$ is empty.

Now it remains to show that $T^{\phi}_{r,-1/2}=0$. If $\lambda_r>\beta_{r-1}$, then $\lambda_1\geq\lambda_2\geq\ldots\geq\lambda_r>\beta_{r-1}\geq\beta_r$, so 
\[
(r-1)\text{th column} = \frac{R(\lambda_{r-1}+1,\beta_{r-1}+1)}{R(\lambda_r,\beta_r)}(r\text{th column}),
\]
implying the determinant is zero. An identical argument holds if $\beta_r>\lambda_{r-1}$.
\end{proof}

For the rest of the proof, assume that $\max(\lambda_r,\beta_r)\leq\min(\lambda_{r-1},\beta_{r-1})$.

Notice now that the determinant in $T_{2r-1}^{\phi}$ is of size $r$, which needs to be compared to a determinant of size $r-1$. To show that the larger determinant is $(1-q)^{2r-2}R(\lambda_r,\beta_r)$ times the smaller determinant, we perform a sequence of operations to the smaller matrix. These operations are slightly different for $\lambda_r>\beta_r$ and $\lambda_r\leq\beta_r$. Consider $\lambda_r>\beta_r$ for now.

First, add a row and a column to the matrix of size $r-1$. The $r$th column is $[0,0,0,\ldots,0,R(\lambda_r,\beta_r)]$ and the $r$th row is $[R(\lambda_r,\beta_1-1+r),R(\lambda_r,\beta_2-2+r),\ldots,R(\lambda_r,\beta_{r-1}+1),R(\lambda_r,\beta_r)]$.
This multiplies the determinant by $R(\lambda_r,\beta_r)$.

Second, for $1\leq i\leq r-1$, perform row operations by replacing the $i$th row with
\[
i\text{th row} + \frac{1}{(q-1)^2}\frac{R(\lambda_i-i+r,\beta_r)}{R(\lambda_r,\beta_r)}(r\text{th row}).
\]
For $1\leq i,j\leq r-1$ and letting $(x,y)=(\lambda_i-i+r,\beta_j-j+r)$, the $(i,j)$ entry is
\begin{eqnarray*}
&&\frac{q^{x+y-2\max(\lambda_r,\beta_r)}-q^{\vert x-y\vert}}{q^2-1}+ \frac{1}{(q-1)^2}\frac{R(x,\beta_r)}{R(\lambda_r,\beta_r)}R(\lambda_r,y)\\
&=&\frac{q^{x+y-2\max(\lambda_r,\beta_r)}-q^{\vert x-y\vert}}{q^2-1}-\dfrac{q^{x-\lambda_r}}{(1+1_{y=0})(q^2-1)}( q^{\lambda_r+y} + q^{\vert \lambda_r-y\vert})\\
&=&\frac{-q^{x+y}-q^{\vert x-y\vert}}{q^2-1}=(1-q)^{-2}R(x,y).
\end{eqnarray*}
Here, we used the fact that $y\geq\beta_{r-1}\geq\min(\lambda_{r-1},\beta_{r-1})\geq\lambda_r$ and $y\geq\lambda_r>\beta_r\geq 0$. For $j=r$, $\lambda_r>y=\beta_r$, so the $(i,j)$ entry is
\[
\frac{-q^{x+y}-q^{\vert x-y\vert}}{(1+1_{y=0})(q^2-1)}=(1-q)^{-2}R(x,y).
\]
Thus, the larger determinant is $(1-q)^{2(r-1)}R(\lambda_r,\beta_r)$ times the larger determinant.

Now consider $\lambda_r\leq\beta_r$. First, add the $r$th row, which is $[0,0,\ldots,0,R(\lambda_r,\beta_r)]$, and add the $r$th column which is $[R(\lambda_1-1+r,\beta_r),R(\lambda_2-2+r,\beta_r),\ldots,R(\lambda_r,\beta_r)]$. This multiplies the determinant by $R(\lambda_r,\beta_r)$. 

Second, for $1\leq j\leq r-1$, perform column operations by replacing tSecond, for $1\leq j\leq r-1$, perform column operations by replacing the $j$th column with
\[
j\text{th column} + \frac{1}{(q-1)^2}\frac{R(\lambda_r,\beta_j-j+r)}{R(\lambda_r,\beta_r)}(r\text{th column}).
\]
Once again, this yields a matrix whose entries are $(1-q)^{-2}R(\lambda_i-i+r,\beta_j-j+r)$, except for the last column, which is $R(\lambda_i-i+r,\beta_r)$.

\section{Asymptotics}\label{4}
Thus far, we have shown that $\tilde{X}^k(n)$ is determinantal with correlation kernel $K(r,a,s_1;r,a,s_2)$. In this section, we will take asymptotics of $K(r_1,a_1,s_1;r_2,a_2,s_2)$, with $(r_1,a_1)$ not necessarily equal to $(r_2,a_2)$. This is because the asymptotic analysis is not much more difficult, and this would be the appropriate limit if Conjecture \ref{Conjecture} were true. Recall that $(r,a)$ corresponds to the $2r+1/2+a$ level. 

\subsection{Discrete Jacobi Kernel}
For $-1< u < 1$ and $a_1,a_2=\pm\half$, define the \textit{discrete Jacobi kernel}
$L(r_1,a_1,s_1,r_2,a_2,s_2;u)$ as follows. If $2r_1+a_1\geq 2r_2+a_2$, then
\begin{multline*}
L(r_1,a_1,s_1,r_2,a_2,s_2;u)\\
=\frac{W^{(a_1,-1/2)}(s_1)}{\pi}\int_u^1 \mathsf{J}_{s_1}^{(a_1,-1/2)}(x)\mathsf{J}_{s_2}^{(a_2,-1/2)}(x)(x-1)^{r_1-r_2}(1-x)^{a_1}(1+x)^{-1/2}dx.
\end{multline*}
If $2r_1+a_1< 2r_2+a_2$, then
\begin{multline*}
L(r_1,a_1,s_1,r_2,a_2,s_2;u)\\
=-\frac{W^{(a_1,-1/2)}(s_1)}{\pi}\int_{-1}^u
\mathsf{J}_{s_1}^{(a_1,-1/2)}(x)\mathsf{J}_{s_2}^{(a_2,-1/2)}(x)(x-1)^{r_1-r_2}(1-x)^{a_1}(1+x)^{-1/2}dx.
\end{multline*}

\begin{theorem}
Let $n$ depend on $N$ in such a way that $n/N\rightarrow t$. Let $r_1,\ldots,r_l$ depend on $N$ in such a way that $r_i/N\rightarrow l$ and their differences $r_i-r_j$ are fixed finite constants. Here, $t,l>0$. Fix $s_1,s_2,\ldots,s_l$ to be finite constants. Let 
\[
\theta=1+\frac{2l}{(l-t)(2\alpha+\alpha^2)}
\]
Then 
\begin{multline*}
\lim_{N\rightarrow\infty}\det[K(r_i,a_i,s_i,r_j,a_j,s_j)]_{i,j=1}^l \\
=\begin{cases}
1,\ \ &l\geq (1-(1+\alpha)^{-2})t\\
\det[L(r_i,a_i,s_i,r_j,a_j,s_j;\theta)]_{i,j=1}^l,\ \ &l<(1-(1+\alpha)^{-2})t
\end{cases}
\end{multline*}
\end{theorem}
\begin{proof}
First consider the case when $l<t$. Let $A(z)=-t\log(1+\alpha(1-z)+\alpha^2/2\cdot(1-z))+l*\log(z-1)$. Then the kernel asymptotically equals
\begin{multline*}
1_{2r_1+a_1\geq 2r_2+a_2}\frac{W^{(a_1,-1/2)}(s_1)}{\pi}\int_{-1}^1\mathsf{J}_{s_1}^{(a_1,-1/2)}(x)\mathsf{J}_{s_2}^{(a_2,-1/2)}(x)(x-1)^{r_1-r_2}(1-x)^{a_1}(1+x)^{-1/2}dx\\
+\frac{W^{(a_1,-1/2)}(s_1)}{\pi}\frac{1}{2\pi i}\int_{-1}^1\oint_C  \frac{e^{N(A(x)-A(\theta))}}{e^{N(A(u)-A(\theta))}}\mathsf{J}_{s_1}^{(a_1,-1/2)}(x)\mathsf{J}_{s_2}^{(a_2,-1/2)}(u)\\
\times (x-1)^{r_1-r_2}\frac{(1-x)^{a_1}(1+x)^{-1/2}dudx}{x-u}.
\end{multline*}
Deform the contours as in Figure \ref{Contours}. With these deformations, the double integral converges to zero, but residues are picked up at $u=x$. For $l>(1-(1+\alpha)^{-2})t$, the parameter $\theta$ is less than $-1$, so no residues are picked up. We arrive at a triangular matrix with diagonal entries equal to $1$, so the determinant is $1$. For $l<(1-(1+\alpha)^{-2})t$, the parameter $\theta$ is in $(-1,1)$, and the residues give the discrete Jacobi kernel. 

\begin{center}
\begin{figure}[htp]
\caption{Shaded regions indicate $\Re(A(z)-A(\theta))>0$ and white regions indicate $\Re(A(z)-A(\theta))<0$. The double zero occurs at $\theta$.}
\begin{center}
\includegraphics[height=2in]{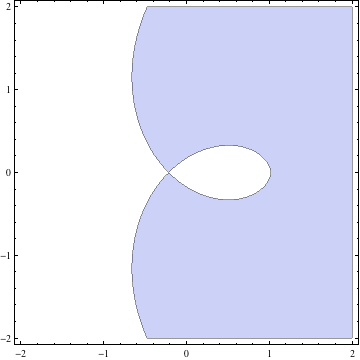}\ \ \ \ \ \  \includegraphics[height=2in]{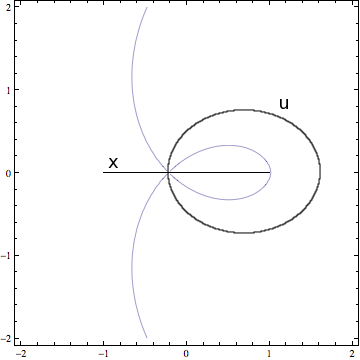}
\end{center}
\label{Contours}
\end{figure}
\end{center}

For $l>t$, the situation is quite different, due to the discontinuity in $\theta$ at $l=t$. Make the substitutions $x=(z+z^{-1})/2$ and $u=(v+v^{-1})/2$. Now the $x$-contour is the unit circle and the $v$ contour is a simple loop that goes outside the unit circle. After deforming as shown in Figure \ref{Contours2}, the double integral converges to $0$, with no residues picked up. Again, we obtain a triangular matrix with diagonal entries equal to $1$.

\begin{center}
\begin{figure}[htp]
\caption{Shaded regions indicate $\Re(A\left(\frac{z+z^{-1}}{2}\right)-A(-1))>0$ and white regions indicate $\Re(A(z)-A(-1))<0$. The double zero occurs at $-1$.}
\begin{center}
\includegraphics[height=2in]{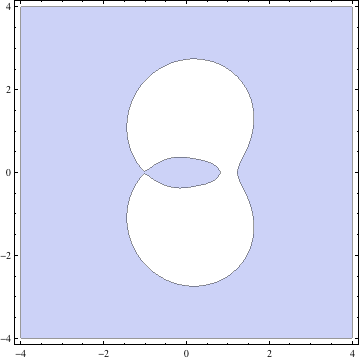}\ \ \ \ \ \  \includegraphics[height=2in]{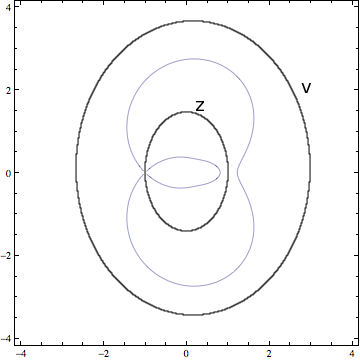}
\end{center}
\label{Contours2}
\end{figure}
\end{center}

\end{proof}

\subsection{Symmetric Pearcey Kernel}\label{SPK}
Define the \textit{symmetric Pearcey kernel} $\mathcal{K}$ on $\R_+\times\R$ as follows. In the expressions below, the $u$-contour is integrated on rays from $\infty e^{i\pi/4}$ to $0$ to $\infty e^{-i\pi/4}$. Let
\begin{multline*}\label{GaussianLikeKernel}
\mathcal{K}(\sigma_1,\eta_1,\sigma_2,\eta_2)= \\
\frac{2}{\pi^2 i} \int\int_0^{\infty}\exp(-\eta_1 x^2 + \eta_2 u^2 + u^4 - x^4)\cos(\sigma_1 x) \cos(\sigma_2 u) \frac{u}{u^2-x^2}dxdu\\
-\frac{1_{\eta_2<\eta_1}}{2\sqrt{\pi(\eta_1-\eta_2)}}\left(\exp\frac{(\sigma_1+\sigma_2)^2}{4(\eta_2-\eta_1)}+\exp\frac{(\sigma_1-\sigma_2)^2}{4(\eta_2-\eta_1)}\right).
\end{multline*}

\begin{theorem}
Let $c_{\alpha}$ be the constant $(1+\alpha)(\alpha(2+\alpha))^{-1/4}$. Let $s_1$ and $s_2$ depend on $N$ in such a way that $s_i/N^{1/4}\rightarrow 2^{-5/4}\sigma_ic_{\alpha}^{-1}>0$ as $N\rightarrow\infty$. Let $n$ and $r_1,r_2$ also depend on $N$ in such a way that $n/N\rightarrow 1$ and $(r_j- (1-(1+\alpha)^{-2})N)/\sqrt{N}\rightarrow 2^{-1/2}\eta_j$. Then
\[
(-2)^{r_2-r_1}(-1)^{s_1-s_2}\frac{N^{1/4}}{c_{\alpha}2^{5/4}}K(r_1,a_1,s_1,r_2,a_2,s_2)\rightarrow \mathcal{K}(\sigma_1,\eta_1,\sigma_2,\eta_2).
\]
\end{theorem}
\begin{proof}
Since the proof is almost identical to the proof of Theorem 5.8 from \cite{kn:BK}, the details will be omitted. The only difference is that now 
\[
A(z)=\log\phi_{\alpha}(z)+(1-(1+\alpha)^{-2})\log(z-1),
\]
with asymptotic expansion
\[
A(z)-A(-1)=-\frac{\alpha(2+\alpha)}{8(1+\alpha)^4}(z+1)^2+O((z+1)^3).
\]
\end{proof}

\bibliographystyle{plain}

\begin{center}
\begin{figure}
\caption{The top figure shows left jumps and the bottom figure shows right jumps. A yellow arrow means that the particle has been pushed by a particle below it. A green arrow means that the particle has jumped by itself. A red line means that the particle has been blocked by a particle below.
In the table, keep in mind that $\xi^k_{(k+1)/2}(n+1/2)$ actually correspond to left jumps, but occur at the same time as the right jumps.
}
\label{Jumping}
\includegraphics[height=2in]{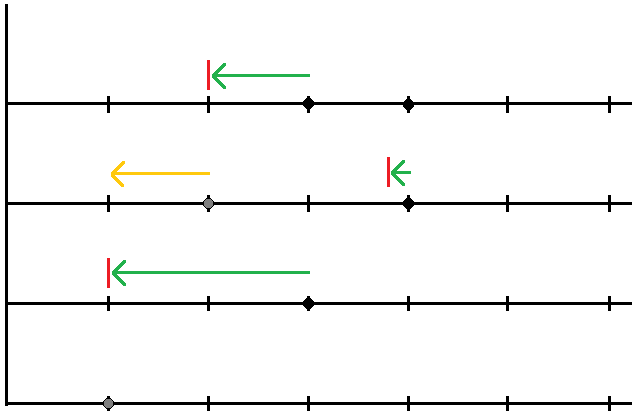}

    \renewcommand\arraystretch{1.33}
    \begin{tabular}{ | l | l | l | l | l | }
    \hline
    $\boldsymbol{\tilde{X}}(n)$ & Left Jumps & $\boldsymbol{\tilde{X}}(n+\half)$ & Right jumps  & $\boldsymbol{\tilde{X}}(n+1)$ \\ \hline
    $\tilde{X}^1_1(n)=1$ &                                     & $\tilde{X}^1_1(n+\half)=1$ & $\xi^1_1(n+\half)=1$ & $\tilde{X}^1_1(n+1)=3$\\  
			    &					    &					 & $\xi^1_1(n+1)=3$ 	 & 	 			\\ \hline
    $\tilde{X}^2_1(n)=3$ & $\xi^2_1(n+\half)=3$ & $\tilde{X}^2_1(n+\half)=1$ & $\xi^2_1(n+1)=1$ 	 & $\tilde{X}^2_1(n+1)=4$ \\ \hline
    $\tilde{X}^3_2(n)=2$ &                                     & $\tilde{X}^3_2(n+\half)=1$ & $\xi^3_2(n+\half)=1$ & $\tilde{X}^3_2(n+1)=0$\\      
			    &					    &					 & $\xi^3_2(n+1)=0$ 	 & 				  \\ \hline
    $\tilde{X}^3_1(n)=4$ & $\xi^3_1(n+\half)=1$ & $\tilde{X}^3_1(n+\half)=4$ & $\xi^3_1(n+1)=0$ 	 & $\tilde{X}^3_1(n+1)=5$  \\    \hline
    $\tilde{X}^4_2(n)=3$ & $\xi^4_2(n+\half)=2$ & $\tilde{X}^4_2(n+\half)=2$ & $\xi^4_2(n+1)=2$ 	 & $\tilde{X}^4_2(n+1)=3$  \\    \hline
    $\tilde{X}^4_1(n)=4$ & $\xi^4_1(n+\half)=0$ & $\tilde{X}^1_1(n+\half)=4$ & $\xi^4_1(n+1)=1$ 	 & $\tilde{X}^4_1(n+1)=6$  \\    \hline
    \end{tabular}
    
\includegraphics[height=2in]{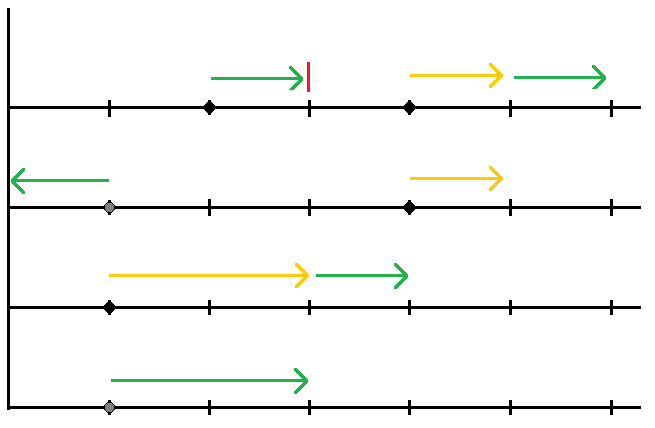}
\end{figure}
\end{center}

\end{document}